\documentclass[12pt]{amsart}
\usepackage{amsfonts}
\usepackage{amsfonts,latexsym,rawfonts,amsmath,amssymb,amsthm}
\usepackage{amsmath,amssymb,amsthm,amscd,esint}
\usepackage[plainpages=false]{hyperref}
\usepackage{appendix}

\usepackage{graphicx}

\RequirePackage{color}

\numberwithin{equation}{section}

\newcommand{\beq}{\begin{equation}}
\newcommand{\eeq}{\end{equation}}
\newcommand{\beqs}{\begin{eqnarray*}}
\newcommand{\eeqs}{\end{eqnarray*}}
\newcommand{\beqn}{\begin{eqnarray}}
\newcommand{\eeqn}{\end{eqnarray}}
\newcommand{\beqa}{\begin{array}}
\newcommand{\eeqa}{\end{array}}

  \newcommand{\cE}{{\mathcal E}}
  \newcommand{\F}{{\mathcal F}}

\newcommand{\C}{{\mathbb C}}
\newcommand{\R}{{\mathbb R}}

\newcommand{\D}{\nabla}

\newcommand{\p}{\partial}
\newcommand{\eps}{\varepsilon}
\newcommand{\lam}{\lambda}
\newcommand{\Om}{\Omega}
\newcommand{\pom}{{\p\Om}}

\newcommand{\tr}{\triangle}
\def\noo{\noindent}

\newtheorem{prop}{Proposition}[section]
\newtheorem{theo}[prop]{Theorem}
\newtheorem{lem}[prop]{Lemma}

\newtheorem{rem}[prop]{Remark}

\allowdisplaybreaks
\title{Moser-Trudinger  inequality  for  the complex Monge-Amp\`ere equation}
\author{Jiaxiang Wang Xu-jia Wang and Bin  $\text{Zhou}^*$}

\address{Jiaxiang Wang: 
School of Mathematical Sciences, Zhejiang University, Hangzhou 310027, China.}
\email{wangjx\underline{ }manifold@126.com}

\address{Xu-jia Wang: 
Centre for Mathematics and Its Applications,
  The Australian National University,
  Canberra, ACT 2601.}
\email{Xu-Jia.Wang@anu.edu.au}

\address{Bin Zhou:
School of Mathematical Sciences, Peking
University, Beijing 100871, China.}
\email{bzhou@pku.edu.cn}

\thanks {*This research is partially supported by ARC DP 170100929 and NSFC 11571018 and 11822101.}



\begin{document}
 \maketitle

\begin{abstract}
In this paper, we prove a Moser-Trudinger type inequality for pluri-subharmonic functions vanishing on the boundary.
Our proof uses a descent gradient flow for the complex Monge-Amp\`ere functional. 
\end{abstract}


\baselineskip=16.4pt
\parskip=3pt
\section{Introduction}

Moser-Trudinger and Sobolev type inequalities
are widely used in the study of partial differential equations and geometric problems.  
The classical Moser-Trudinger inequality in the real two dimensional space was obtained by Trudinger \cite{Tr} and Moser \cite{M}. 
Let $\Omega$ be  a bounded domain in $\mathbb R^2$. Then one has
$$\int_\Omega e^{4\pi\frac{- u}{\|\nabla u\|_{L^2(\Omega)}}}\,dx\leq C\ \ \ \forall\ u\in W^{1,2}_0(\Omega) , $$
where  $C>0$ is a constant depending only on the diameter of the domain.
It is natural to ask whether this inequality can be extended to the complex setting in high dimensions. 

Let $\Omega$ be a bounded, smooth pseudo-convex domain in $\mathbb{C}^n$.
Denote by $\mathcal{PSH}(\Omega)$
the set of pluri-subharmonic functions 
and by $\mathcal{PSH}_0(\Omega)$ the set of functions in $\mathcal{PSH}(\Omega)$ 
which vanish on $\partial \Omega$. For $u\in \mathcal{PSH}_0(\Omega)\cap C^{\infty}(\bar \Omega)$, let 
\beq
\mathcal E(u)=\frac{1}{n+1}\int_{\Omega}(-u)(dd^cu)^n
\eeq
be the Monge-Amp\`ere energy.
Denote
\beq
\|u\|_{\mathcal{PSH}_0(\Omega)}=[\mathcal E(u)]^{\frac{1}{n+1}},
\eeq
which is a semi-norm in the set $\mathcal{PSH}_0(\Omega)$ \cite {W1}.
The main result of the paper is the following Moser-Trudinger type  inequality. 

\begin{theo} \label{mt}
There exist positive constants $\alpha, C$, depending only on $n$  and $\text{diam}(\Omega)$, 
	such that  $\forall\ u\in \mathcal{PSH}_0(\Omega)\cap C^{\infty}(\bar \Omega)$, $ u\not\equiv 0$,
	we have the inequality
	\beq\label{MT0}
	\int_\Omega e^{ \alpha \left(\frac{- u}{\|u\|_{\mathcal{PSH}_0(\Omega)}}\right)^{\frac{n+1}{n}}} \leq C . 
	\eeq
\end{theo}

The Moser-Trudinger type  inequality in the complex space was conjectured by Aubin \cite{Au}
more than three decades ago. 
It has been studied by many authors using various different methods. 
See \cite{Ce, GKY}  for the inequality in complex space and \cite {TZ, PSS, Be} on K\"ahler manifolds. 
But the inequalities obtained in the above mentioned papers are only weak form of the conjectured one.
In this paper we prove the above stronger inequality \eqref{MT0}.
We point out that inequality \eqref{MT0} was also obtained in \cite {BB} many years ago.
The proof in \cite{BB} uses the pluripotential theory and thermodynamical formalism and is complicated.
In this paper we provide a proof based on classical PDE techniques only,
by employing a descent gradient flow method for the associated functional.

Our motivation is to establish the a priori estimates for the complex Monge-Amp\`ere equation
by the classical PDE techniques, and thereby answering a question raised  in \cite{B, BGZ, L}. 
Using the inequality \eqref{MT0},   we prove in a subsequent paper \cite{WWZ}
the uniform estimate, the stability, and  the H\"older continuity of solutions to the complex Monge-Amp\`ere equation 
by traditional PDE techniques.  Inequality \eqref{MT0} will play a key role in our treatment.

The organization of this paper is as follows. 
In Section 2, 
we prove a Sobolev type inequality for the complex Monge-Amp\`ere equation. To prove this inequality, 
we use the gradient flow method to prove that the Sobolev constant is monotonic with respect to domains, 
as in \cite{W1} for $k$-Hessian equations.
However there is an essential difference between our proof here and that in \cite{W1}.
In \cite{W1} the moving plane method is used to prove that 
the solution is rotationally symmetric when the domain is a ball. 
But this technique does not apply to the complex Monge-Amp\`ere equation.
To avoid this obstacle, 
we combine the gradient flow with an induction on the dimension via the Brezis-Merle type inequality,
as described in Step 1 to Step 4 in Section 2. The use of Brezis-Merle type inequality was inspired by \cite{BB}.
In Remark \ref{proof2} we point out that there is a simple proof for the Sobolev type inequality 
if one uses an estimate in \cite{K98}, which was known to the authors many years ago. 
In Sectin 3, we establish a generalized Brezis-Merle type inequality (Theorem \ref{BMq}).
In Sections 4, we prove the Moser-Trudinger type inequality \eqref{MT} by the using Theorem \eqref{BMq} and another gradient flow.

\newpage

\section{Sobolev type inequalities}

In this section, we drive a Sobolev type inequality \eqref{up} for the complex Monge-Amp\`ere equation.
Let us first recall the existence and regularity of solutions
to the associated parabolic Monge-Amp\`ere equation.

Let $\Omega\subset \C^n$ be a bounded, strictly pseudo-convex domain with smooth boundary $\p \Omega$. 
Denote $Q_T:=\Omega\times [0,T)$, 
$Q:=\Omega\times (0,\infty)$, $\Gamma_T:=\p Q_T\setminus \Omega\times\{t=T\}$ and $\Gamma:=\p Q$. 
Consider the Dirichlet problem 
\begin{align}\label{para-app}
\begin{cases}
u_t-\log \det(u_{i\bar{j}})=g(z,t,u),\ \ \ &\text{in\ $Q_T$},   \\
u=\varphi,\ \ \                              &\text{on\ $\p Q_T$},
\end{cases}
\end{align}
where $g\in C^2(\bar{Q}_T\times \R)$,
$\varphi\in C^{4,3}(\bar{Q}_T)$, $\varphi(z,0)$ is pluri-subharmonic,  and satisfies the compatibility condition
\beq\label{phit}
\varphi_t-\log\det(\varphi_{i\bar{j}})=g(z,t,\varphi)\ \ \text{on}\  \p \Omega\times \{t=0\}. 
\eeq
We also assume that there exists $C_1>0$ such that 
\beq\label{gzts}
g(z, t, s)\geq -C_1\ \ \forall\ (z, t)\in Q_T,  \ s\in\R.
\eeq
Note that \eqref{para-app} is parabolic with respect to $u$ if for any given $t>0$, $u(\cdot,t)$ is pluri-subharmonic.

\begin{theo}\label{para-estimate}
Let $u(z, t)$ be a solution to \eqref{para-app}.
Then under the above conditions, we have the estimate
\beq\label{HL0}
\|u\|_{C^{2,1}(\bar Q_T)}\leq C,
\eeq
where $C$ depends only on $\Omega$, $\|\varphi\|_{C^{4,2}(\bar{Q}_T)}$, $C_1$ and $g$ up to its second derivatives. 
\end{theo}

We refer the reader to \cite{HL} for a proof.

Since the a priori estimate \eqref{HL0} is independent of $T$,
by Krylov's regularity theory, 
we obtain the global existence and regularity of solutions to \eqref{para-app}. 

\begin{theo}\label{para-exist}
Under the assumptions of Theorem \ref{para-estimate}, 
there exists a solution $u\in C^{3+\alpha, 2+\frac{\alpha}{2}}(\bar Q)$ to \eqref{para-app}.
\end{theo}

The Sobolev inequality for the complex Monge-Amp\`ere equation is as follows.

\begin{theo} \label{sobolev}
Let $\Omega$ be a bounded, smooth, pseudo-convex domain.
Then for any $p>1$,
\beq\label{up}
\|u\|_{L^{p}(\Omega)}\leq C \|u\|_{\mathcal{PSH}_0(\Omega)}, 
      \  \ \ \forall\ u\in \mathcal{PSH}_0(\Omega)\cap C^{\infty}(\bar \Omega) ,
\eeq
where $C$ depends on $n$, $p$ and $\text{diam}(\Omega)$.
\end{theo}

\noindent{\it Proof of Theorem \ref{sobolev}.} 
Note that inequality \eqref{up} in the case $p\leq n$ can be established by a direct computation \cite{CP}.
In the following we will assume that $p>n$. 
Denote
\beq\label{ratio}
T_{p, \Om} =: \inf_{u\in \mathcal{PSH}_0(\Omega)} \frac {\mathcal E(u)}{\|u\|^{n+1}_{L^{p+1}(\Om)}} .
\eeq
It suffices to prove
\beq\label{TT2}
T_{p, \Om}\ge  \lam  
\eeq
for some small constant $\lam>0$. 
We divide the proof  into four steps.

\noindent{\it Step 1:} We show that the Sobolev inequality on balls of dimension $n-1$ 
implies an weak Brezis-Merle type inequality on balls of dimension $n$,  
in terms of the {\it complex Monge-Amp\`ere mass} 
\begin{eqnarray*}
\mathcal M(u)&=&\int_\Omega (dd^cu)^n
=\frac{1}{ n!}\int_{\Omega}\det(u_{i\bar j}).
\end{eqnarray*}

\begin{lem}\label{BMS-1}
Assume that the Sobolev type inequality 
\beq
\|u\|_{L^p(B)} \leq C_0\|u\|_{\mathcal{PSH}_0(B)} \ \ \ \forall\ u\in \mathcal{PSH}_0(B)\cap C^{\infty}(\bar B)
\eeq
holds for $p>0$ on any ball $B=B^{(n-1)}_r\subset \mathbb C^{n-1}$,
where $C_0$ depends on $n, p$, 
and the upper bound of the radius $r$. Then the following inequality 
\beq\label{BMS}
\Big(\int_\Omega |u|^{p}\Big)^{\frac{1}{p}}
  \leq \tilde C\cdot C_0 \Big(\int_\Omega (dd^cu)^{n}\Big)^{\frac{1}{n}}
\ \ \ \forall\ u\in \mathcal{PSH}_0(\Omega)\cap C^{\infty}(\bar \Omega),
\eeq
holds on any ball $\Omega\subset \mathbb C^{n}$ of the same radius $r$,
where $\tilde C$ depends on the radius but is independent of $p$.
\end{lem}

\begin{proof}
Without loss of generality, we assume the balls are of radius $1$ and all centered at the origin. 
Write $z=(w,\xi)\in \mathbb C^{n-1}\times \mathbb C$. 
Let $D$ be the disk in $\mathbb C$ of the same radius $r$. 
For any $\xi=t+\sqrt{-1}s\in D$, denote
$D_\xi:=\{w\in \mathbb C^{n-1}\ | \ |w|^2\leq 1-|\xi|^2\}$. 
For $u(z)\in \mathcal{PSH}_0(\Omega)\cap C^{\infty}_0(\bar\Omega)$, 
we denote 
$$v(\xi)= \int_{D_\xi} (-u)(d_wd^c_wu)^{n-1}.$$

First we claim:
\beq\label{diff}
\int_D|-\tr_\xi v(\xi)|\,dt\,ds\leq 2\int_{\Omega}(dd^cu)^n.
\eeq
Here $\tr_\xi=\frac{1}{\sqrt{-1}} d_\xi d_\xi^c=\frac{\p^2}{\p t^2}+\frac{\p^2}{\p s^2}$.
In order to compute $\tr_\xi v$, we use the spherical coordinates $(r, \theta)$ on $\mathbb C^{n-1}$. 
For convenience, we denote $(-u)\det(u_{w^i\bar{w}^j}):=F(w, \xi)=F(r, \theta, \xi)$, where $r=|w|$.
Then 
\beqs
v(\xi)=\int_0^{\sqrt{1-t^2-s^2}}\int_{S^{2n-3}}F(r, \theta, \xi) \cdot r^{2n-3}\,d\sigma\,dr,
\eeqs 
where $d\sigma$ is the standard measure on $S^{2n-3}$. 
Hence, we have 
\beqs
\frac{\p v}{\p t} &=&\int_{S^{2n-3}}[F(r, \theta, \xi)\cdot r^{2n-3}]\big |_{r=\sqrt{1-t^2-s^2}} \,d\sigma \cdot \frac{\p \sqrt{1-t^2-s^2}}{\p t}\\
&&+\int_0^{\sqrt{1-t^2-s^2}}\int_{S^{2n-3}}\frac{\partial}{\partial t}[F(r, \theta, \xi)\cdot r^{2n-3}]\,d\sigma\,dr\\
&=&\int_0^{\sqrt{1-t^2-s^2}}\int_{S^{2n-3}}\frac{\partial}{\partial t}F(r, \theta, \xi) \cdot r^{2n-3}\,d\sigma\,dr.
\eeqs
Here the first term vanishes by $u\big|_{\p B}=0$. The second derivative 
\beqs
\frac{\p^2 v}{\p t^2} &=&\int_{S^{2n-3}}\frac{\partial}{\partial t}F(r, \theta, \xi)\big |_{r=\sqrt{1-t^2-s^2}}  \cdot (\sqrt{1-t^2-s^2})^{2n-3}]\,d\sigma \cdot \frac{\p \sqrt{1-t^2-s^2}}{\p t}\\
&&+\int_0^{\sqrt{1-t^2-s^2}}\int_{S^{2n-3}}\frac{\partial^2}{\partial t^2}F(r, \theta, \xi)\cdot r^{2n-3}\,d\sigma\,dr.
\eeqs
Similarly, we can compute $\frac{\p^2 v}{\p s^2}$. Hence,
\beq\label{la}
\tr_\xi v=\int_0^{\sqrt{1-t^2-s^2}}\int_{S^{2n-3}}\tr_\xi F(r, \theta, \xi) \cdot r^{2n-3}\,d\sigma\,dr
+\int_{S^{2n-3}}G(\theta,\xi)\,d\sigma, 
\eeq
where 
$$G(\theta,\xi)=\frac{\partial F}{\partial t}\big |_{r=\sqrt{1-t^2-s^2}}\cdot\frac{\p \sqrt{1-t^2-s^2}}{\p t}+\frac{\partial F}{\partial s}\big |_{r=\sqrt{1-t^2-s^2}}\cdot\frac{\p \sqrt{1-t^2-s^2}}{\p s}$$
is the outward normal derivative of $(-u)\det (u_{w^i\bar{w}^j})$ on the boundary of $D_\xi$ (for a fixed $\xi$).  
Note that $(-u)\det (u_{w^i\bar{w}^j})>0$ in $D_\xi$ and $(-u)\det (u_{w^i\bar{w}^j})=0$ on the boundary. We have $G(\theta,\xi)\leq 0$.  

By the divergence theorem and $\D_\xi v\big|_{\p D}=0$, we have 
\beqs
0&= & \int_{D}-\tr_\xi v(\xi)\,d\mu_\xi   \\
&= &  -\int_{D}\int_0^{\sqrt{1-t^2-s^2}}\int_{S^{2n-3}}\tr_\xi F(r, \theta, \xi) \cdot r^{2n-3}\,d\sigma\,dr
- \int_{D}\int_{S^{2n-3}}G(\theta,\xi)\,d\sigma\,d\mu_\xi.
\eeqs
Changing the coordinates back to $(w,\xi)$, the first term in \eqref{la}
equals $\int_{\Omega}(dd^cu)^n$.
By the non-positivity of $G$, $$ \int_{D}\left|\int_{S^{2n-3}}G(\theta,\xi)\,d\sigma\right|\,d\mu_\xi\leq \int_{\Omega}(dd^cu)^n.$$
The claim is proved.

By the Sobolev inequality in dimension $n-1$,
\begin{eqnarray*}
\Big(\int_\Omega |u|^{p}\Big)^{\frac{1}{p}}&=&\Big (\int_{|\xi|^2\leq 1}\int_{D_\xi} |u|^{p}\,d\mu_w \,d\mu_\xi\Big )^{\frac{1}{p}} \\
&\leq& C_0\Big (\int_{|\xi|^2\leq 1}\Big (\int_{D_\xi}(-u)\det(u_{w^i\bar w^j})\Big )^{\frac{p}{n}} d\mu_\xi\Big )^{\frac{1}{p}} \\
&=& C_0\Big (\int_{|\xi|^2\leq 1} [v(\xi)]^{\frac{p}{n}}\,d\mu_\xi\Big )^{\frac{1}{p}}\\
&\leq & C\cdot C_0\Big (\int_{|\xi|^2\leq 1} |-\triangle_\xi v(\xi)|\Big )^{\frac{1}{n}}\\
 & \leq & C\cdot C_0\Big (\int_\Omega (dd^cu)^{n}\Big )^{\frac{1}{n}}.
\end{eqnarray*}
The Brezis-Merle inequality in real dimension $2$(see \eqref{laplace}) is used in the last inequality.
\end{proof}

\noindent{\it Step 2:} 
We show that the weak Brezis-Merle type inequality in dimension $n$ for any smooth pseudo-convex 
domain $\Omega\subset\mathbb C^n$ implies the Sobolev type inequality in the same dimension. 
So we first assume the following inequality
\beq\label{BMS1} 
\Big (\int_\Omega |u|^{p}\Big )^{\frac{1}{p}}\leq C_1 \Big (\int_\Omega (dd^cu)^{n}\Big )^{\frac{1}{n}}, \ 
\ \ \forall\ u\in \mathcal{PSH}_0(\Omega)\cap C^{\infty}(\bar \Omega) 
\eeq
holds for some $C_1>0$ depending only on $n$, $diam(\Omega)$.

As in \cite{W1}, we denote 
\beq\label{f(t)}
f(t)=\Big\{ {\begin{split}
&|t|^p\ \ \hskip18pt \ \ |t|\le M,\\
&e^{-M} t^{-2}\ \ \ |t|\ge M+e^{-M},
\end{split}} 
\eeq
where $M>1$ is a large constant; and
denote 
\beq\label{F-1}
J(u) = \int_\Omega (-u)  \det(u_{i\bar j})  \,dV
             -  \lam\Big[ (p+1) \int_\Omega F(u) \,dV\Big]^{\frac{n+1}{p+1}},
\eeq
where $F(u)=\int_0^{|u|} f(t)\,dt$.
If \eqref{TT2} is not true, then for a small $\lambda>0$ and large $M$,  
we have  
\beq\label{111}
\inf_{u\in \mathcal{PSH}_0(\Omega)\cap C^{\infty}(\bar \Omega)} J(u)<-1.
\eeq

To make the equation non-degenerate, 
we replace $f$ by $f_\delta=f+\delta$ for a small $\delta>0$ and consider 
\beq\label{F-1}
J_\delta(u) = \int_\Omega (-u)  \det(u_{i\bar j})  \,dV
             -  \lam\Big[ (p+1) \int_\Omega F_\delta(u) \,dV\Big]^{\frac{n+1}{p+1}},
\eeq
where $F_\delta(u)=\int_0^{|u|} f_\delta(t)\,dt$.
Then for sufficiently small $\delta>0$, we still have
\beq\label{111d}
\inf_{u\in \mathcal{PSH}_0(\Omega)\cap C^{\infty}(\bar \Omega)} J_\delta(u)<-1.
\eeq
 Now consider the parabolic equation
\beq\label{PE}
\Bigg\{
{\begin{split}
 &u_t - \log\det(u_{i\bar j}) = -\log \lam \beta_\delta(u) f_\delta(u)\ \ \ \text{in}\ Q=\Omega\times (0, \infty),\\
 &u(x, 0)= w_\eps, \ \ \ \text{and}\ \  \\
 & u=0 \hskip26pt \ \text{on}\ \p \Omega\times (0, \infty),\
 \end{split} } \eeq
where $w_\eps$ is chosen such that
$$J_\delta(w_\eps)\le \inf_{u\in \mathcal{PSH}_0(\Omega)\cap C^{\infty}(\bar \Omega)} J_\delta(u)+\eps<-1,$$ 
and 
$$\beta_\delta(u)=\Big[(p+1) \int_\Omega F_\delta(u) \,dV\Big]^{\frac {n-p}{p+1}}.$$
\eqref{PE} is a descent gradient flow for the functional $J_\delta$. 
By Theorem \ref{para-exist}, 
there exists a solution $u_{\delta}(z, t)\in C^{3+\alpha,2+\frac{\alpha}{2}}(\bar Q)$. 
To apply Theorem \ref{para-exist}, we need the compatibility condition \eqref{phit}. 
In order that $w_\eps$ satisfies \eqref{phit}, 
we can modify $w_\eps$ slightly near the boundary, by solving the Dirichlet problem
$\det w_{i\bar j}=g$ in $\Omega$, $w=0$ on $\p\Omega$, where $g\in C^2(\overline\Omega)$, 
$g=\det (w_\eps)_{i\bar j}$ in $\Omega_\sigma :=\{z\in\Om\ |\ \text{dist}(z, \p\Omega)>\sigma\}$
and $g=1$ on $\p\Omega$, for a sufficiently small $\sigma>0$.


Note that 
$$\frac{d}{dt}J_\delta(u(\cdot, t))=-\int_\Omega[\det (u_{i\bar j})-\lambda\beta_\delta(u)f_\delta(u)]u_t \, dV\leq 0$$
along the flow.
Hence, there exists a sequence $t_j\to\infty$ so that $\frac{d}{dt}J_\delta(u_{\delta}(\cdot, t))\to 0$. 
Since $u_{\delta}(\cdot, t_j)$ are uniformly bounded, there exists a subsequence which
converges to a solution $v_{\delta}(z)\in C^2(\bar \Omega)$ to 
\beq\label{DP1d}
{\begin{split}
 \det(v_{i\bar j}) & = \lambda\beta_\delta(v) f_\delta(v) \hskip10pt \text{in}\  \Omega, \\
   v &=0\ \hskip55pt \text{on}\ \p \Omega,
   \end{split}}
\eeq
with $J_\delta(v_\delta)\leq -1$. Note that by \eqref{111d}, $\beta_\delta(v_\delta)\leq C\lambda^{\frac{p-n}{n+1}}$.

\noo {\bf Claim}: For sufficiently large $M$ we have, 
\begin{align}
 f_{\delta}(v_\delta)= & (1+o(1))\left(|v_\delta|^{p}+\delta\right); \label{beta1}\\
\beta_\delta(v_\delta)= & (1+o(1)) \Big[ \int_\Omega |v_\delta|^{p+1}\,d\mu\Big]^{(n-p)/(p+1)}.\label{beta}
\end{align}

To prove the first statement \eqref{beta1} in claim, 
denote $\Omega^*=\{z\in \Omega: v_\delta(x)\leq -M-e^{-M}\}$. Write the equation as
$$V^{i\bar j}(v_\delta)_{i\bar j}=n\lambda \beta_{\delta}(v_\delta)f_{\delta}(v_\delta),$$
where $(V^{i\bar j})$ is the cofactor matrix of $((v_\delta)_{i\bar j})$. 
By Aleksandrov's maximum principle, 
\begin{eqnarray*}
  \inf_\Omega v_\delta&\geq& \inf_{\partial \Omega^*}v_\delta-C\left[\int_{\Omega^*}
     \frac{|n\lambda\beta_{\delta}(v_\delta)f_{\delta}(v_\delta)|^{2n}}{[\det (V^{i\bar j})]^2}\,d\mu\right]^{\frac{1}{2n}}\\
&=&\inf_{\partial \Omega^*}u_\delta-nC\left[\int_{\Omega^*}[\lambda\beta_{\delta}(v_\delta)f_{\delta}(v_\delta)]^2\,d\mu\right]^{\frac{1}{2n}}\\
&\geq & -M-C(1+\delta)e^{-\frac{M}{2n}}-C\delta^2.
\end{eqnarray*}
Here $C$ depends on $\text{diam}(\Omega)$, $n$ and $\lambda$.
Hence, \eqref{beta1} holds by choosing $M$  sufficiently large.
The second statement \eqref{beta} in the claim can be derived from the first one. 
Since $(1+x)^{-\alpha}\geq 1-\alpha x$ when $0<\alpha<1$,  by a direct computation and noting that $p>n$, we have
$$C_2\|v_\delta\|_{L^{p+1}(\Omega)}^{n-p}-C_2\frac{p-n}{p+1}\delta \|v_\delta\|_{L^1(\Omega)}\|v_\delta\|_{L^{p+1}(\Omega)}^{n-2p-1} \leq \beta_{\delta}(v_\delta)
      \leq  C_2\|v_\delta\|_{L^{p+1}(\Omega)}^{n-p},$$ 
where $C_2$ is independent of $p$. Then \eqref{beta} follows. 

Now by \eqref{beta1},
$$\int_\Omega f_{\delta}(v_\delta)\leq C_3\Big[\int_\Omega (-v_\delta)^{p+1}\,d\mu\Big]^{\frac{p}{p+1}}+\delta|\Omega|,$$
where $C_3$ depends only on $|\Omega|$. Since $|\Omega|\leq C\cdot \text{diam}(\Omega)^{2n}$, the constant above depends only on the diameter of the domain.
By \eqref{BMS1}, 
we have
\begin{eqnarray}\label{computation}
\|v_\delta\|_{L^{p+1}}(\Omega) 
  &\leq& C_1\Big[\int_\Omega (dd^cv_\delta)^{n}\Big]^{\frac{1}{n}}\nonumber\\
   &=&C_1\Big (\int_\Omega \lambda\beta_{\delta}(v_\delta)f_{\delta}(v_\delta)\,d\mu\Big )^{\frac{1}{n}}\nonumber\\
  &\leq& C_1\lambda^{\frac{1}{n}}\beta_{\delta}^{\frac{1}{n}}\big(C_3\|v_\delta\|_{L^{p+1}(\Omega)}^{p}
   +\delta\cdot\text{diam}(\Omega)^{2n}\big)^{\frac{1}{n}}.
\end{eqnarray}
By the a priori estimates in \cite{CKNS}, $v_{ \delta}$ converges to a function $u\in C^{3,1}(\Omega)\cap C^{1,1} (\bar\Omega)$ which solves
\beq\label{DP1}
{\begin{split}
 \det(u_{i\bar j}) & = \lambda\beta(u) f(u) \hskip30pt \text{in}\  \Omega, \\
   u &=0\ \hskip70pt \text{on}\ \p \Omega,
   \end{split}}
\eeq
with $J_\delta(u)\leq -1$ as $\delta\to 0$.  
Here  $\beta(\cdot)=\beta_\delta(\cdot)|_{\delta=0}$.
By the Azela-Ascoli Theorem, $\|v_{\delta}\|_{L^{p+1}(\Omega)}\to \|u\|_{L^{p+1}(\Omega)}$ and $\beta_{\delta}(v_{\delta}) \to \beta(u)\sim\|u\|_{L^{p+1}(\Omega)}^{n-p}$. By \eqref{computation}, 
we get $\lambda\geq C\cdot C_1^{-n}$,
where $C$ is independent of $p$. This is a contradiction when $\lambda$ is chosen small.


\vskip 15pt

\noindent{\it Step 3:}
In this step, we suppose that $\Omega_1$ is a smooth, strictly pseudo-convex domain,
and $\Omega_2$ is a ball with $\Omega_1\subset\Omega_2\subset\mathbb C^n$. 
We show
\beq\label{mono}
T_{p,\Omega_1}\geq T_{p,\Omega_2}.
\eeq

Suppose to the contrary that $T_{p,\Omega_1}< T_{p,\Omega_2}$. 
Choose $\lambda\in (T_{p,\Omega_1}, T_{p,\Omega_2})$. 
Let $J(u,\Omega)$ be the functional given in \eqref{F-1}. Then we have
\begin{eqnarray*}
&&\inf\{J(u,\Omega_1): u\in \mathcal{PSH}_0(\Omega_1)\cap C^{\infty}(\bar \Omega_1)\} <-1,\\
&&\inf\{J(u,\Omega_2): u\in \mathcal{PSH}_0(\Omega_2)\cap C^{\infty}(\bar \Omega_2)\} \geq 0 
\end{eqnarray*}
when $M>>1$. Suppose $u_1$ is the solution to \eqref{DP1} on $\Omega_1$ obtained as Step 1. 

Extend $u_1$ to $\Omega_2$ such that $u_1=0$ on $\Omega_2\setminus\bar\Omega_1$
(so $u_1$ is not pluri-subharmonic in $\Omega_2$). Extend $\Psi$ to on $\Omega_2$ by $\Psi=0$ in $\Omega_2\setminus\bar\Omega_1$.
Denote
\beq\label{F-3}
E(u) = \int_{\Omega_2} (-u)  \Psi  \,d\mu
             -  \lam\Big[ \int_{\Omega_2} |u|^{p+1} \,d\mu\Big]^{\frac{n+1}{p+1}}.
\eeq
Since $u_1=0$ outside $\Omega_1$, $E(u_1)\leq J(u_1,\Omega_1)\leq -1$.

Let $\Psi_k$ be a sequence of smooth, monotone decreasing approximation of $\Psi$ such that 
$$\Psi_k>0,\ \sup_{\Omega_1}|\Psi_k-\Psi|<e^{-k},\ \int_{\Omega_2}|\Psi_k-\Psi|^2\,d\mu<e^{-2k}.$$ 
Let $u_2=u_{2,k}$ be the solution to
\beq\label{DP2}
{\begin{split}
 \det(u_{i\bar j}) & = \Psi_k \hskip20pt \text{in}\  \Omega_2, \\
   u &=0\ \hskip20pt  \text{on}\ \p \Omega_2.
   \end{split}}
\eeq
By \cite{CP} (note that when $\Omega_2$ is a ball, it is a direct computation), we have $$\int_{\Omega_2}(-u_2)\leq C\|\Psi_k\|_{L^2(\Omega_2)}^{\frac{1}{n}}$$
for a constant $C$ depending on $n$ and $\text{diam}(\Omega_2)$. Then we can choose $k$ sufficiently large such that 
by the H\"older inequality,
$$\int_{\Omega_2}(-u_2)\left(\det(u_{2,i\bar{j}})-\Psi\right)<\frac{1}{2}.$$

By the comparison principle, $u_2<u_1\leq 0$ in $\Omega_1$ and $u_2$ is uniformly bounded.
\begin{eqnarray*}\label{F-4}
E(u_2) &=& \int_{\Omega_2} (-u_2)  \Psi  \,dV
             -  \lam\Big[ \int_{\Omega_2} |u_2|^{p+1}\,d\mu\Big]^{\frac{n+1}{p+1}}\\
&\geq& \int_{\Omega_2} (-u_2)  \det(u_{2,i\bar j}) dV
             -  \lam\Big[ \int_{\Omega_2} |u_2|^{p+1}\,d\mu\Big]^{\frac{n+1}{p+1}}-\frac{1}{2}\\
&\geq& -\frac{1}{2}.
\end{eqnarray*}
Denote $\Phi(t)=E(u_1+t(u_2-u_1))$. Then $\Phi(0)=E(u_1)\leq -1$, $\Phi(1)\geq -\frac{1}{2}$.
By direct computation,
\begin{eqnarray*}
\Phi'(0)=\int_{\Omega_2} (u_1-u_2)  \Psi  \,d\mu
             -  (n+1)\lam\Big[ \int_{\Omega_2} |u_1|^{p+1}\,d\mu\Big]^{\frac{n+1}{p+1}} 
\int_{\Omega_2} |u_1|^{p}(u_1-u_2)\,d\mu.
\end{eqnarray*}
Note that by Step 2,
$$\beta(u_1) =(1+o(1)) \Big[ \int_\Omega |u_1|^{p+1}\,d\mu\Big]^{(n-p)/(p+1)}.$$
It follows that 
\begin{eqnarray*}
\int_{\Omega_1} (u_1-u_2)  \det(u_{1,i\bar j})  \,d\mu
&=&  (1+o(1))\lam\Big[ \int_{\Omega_1} |u_1|^{p+1}\,d\mu\Big]^{\frac{n+1}{p+1}} 
\int_{\Omega_1} |u_1|^{p}(u_1-u_2)\,d\mu\\
             &<&   (n+1)\lam\Big[ \int_{\Omega_1} |u_1|^{p+1}\,d\mu\Big]^{\frac{n+1}{p+1}} 
\int_{\Omega_1} |u_1|^{p}(u_1-u_2)\,d\mu.
\end{eqnarray*}
This implies $\Phi'(0)<0$. Note that the functional $E$ is linear in the first integral and convex in the second integral, we have  $\Phi''(t)<0$ for  $t\in (0,1)$. Therefore we have $\Phi(1)<\Phi(0)$. We reach a contradiction.  Hence, \eqref{mono} holds.

\vskip 15pt

\noindent{\it Step 4:} We finish the proof by an induction on the dimensions, 
using the results in the first three steps. 
By the classical Sobolev inequality in real dimension 2, i.e., complex dimension $1$ and Lemma \ref{BMS-1}, \eqref{BMS} holds in
complex dimension $2$. Note that the constant depends on the upper bound of the domain. 
Then by Step 2, we have Sobolev inequality for any ball in $\mathbb C^2$. The Sobolev inequality for general smooth strictly pseudo-convex domain $\Omega\subset\mathbb C^2$
follows by Step 3. By induction in this way, we obtain the Sobolev inequality in all dimensions.
\hfill$\square$

\begin{rem}\label{wwMT}
As with the k-Hessian equation \cite {TiW},
a weak Moser-Trudinger type inequality can be obtained by using the Sobolev inequalities and Taylor's expansion as follows,
Let $C_{n,p}$ be the Sobolev constant in dimension $n$, i.e.,
$$\|u\|_{L^{p}(\Omega)}\leq C_{n,p}\cdot\|u\|_{\mathcal{PSH}_0(\Omega)}.$$
Equivalently, one has
$$\int_\Omega\left(\frac{|u|}{\|u\|_{\mathcal{PSH}_0(\Omega)}}\right)^{p}\,d\mu\leq C_{n,p+1}^{p}.$$
By checking the proof of Theorem \ref{sobolev}, we have $C_{n, p}\leq \tilde C \cdot C_{n-1,p}$
for some constant $\tilde C$ independent of $p$.
Hence, by Taylor's expansion,  there exists $\alpha>0$ such that
\beq\label{wmt}
\int_\Omega e^{\alpha\frac{-u}{\|u\|_{\mathcal{PSH}_0(\Omega)}}}\,d\mu
=\int_\Omega \sum_{j=1}^\infty \frac{1}{j!}\left(\alpha\frac{-u}{\|u\|_{\mathcal{PSH}_0(\Omega)}}\right)^j\,d\mu\leq C
\eeq
for some $C>0$.
Inequality \eqref{wmt} was recently obtained for plurisubharmonic functions with finite pluricomplex energy in \cite{Ce}.
\end{rem}

\begin{rem}\label{proof2}
As  mentioned in the introduction, 
there is a simple proof for the Sobolev type inequality  \eqref{up}
if we use Ko\l{}odziej's $L^\infty$-estimate. 
In fact, to prove  \eqref{up}, Steps 1, 3 and 4 in Section 2 are not needed 
if we use Ko\l{}odziej's estimate \cite{K98}. 
The purpose of Steps 1, 3 and 4 is to prove \eqref{BMS1}, 
which can be replaced by Ko\l{}odziej's $L^\infty$ estimate.
More precisely, 
in Step 2, we use the  gradient flow argument  \eqref{PE}.
It suffices to  prove $\lambda\geq c>0$ in \eqref{DP1d}  for some $c>0$ independent of $M$. 
Note that \eqref{BMS1} was used in \eqref{computation} only, 
which can be replaced by Ko\l{}odziej's $L^\infty$ estimate as follows:
\begin{eqnarray}\label{computation1}
\|v_{\delta}\|_{L^{\infty}} 
& \le& C\|\lambda \beta_{\delta}f_{\delta}\|_{L^{1+\eps}(\Om)}^{\frac{1}{n}} \\
&  \le& C\lambda^{\frac{1}{n}}\beta_{\delta}^{\frac{1}{n}}\left(\int_\Om (|v_{\delta}|+\delta)^{(1+\eps)p}\right)^{\frac{1}{(1+\eps)n}} \nonumber \\
& \le& C\lambda^{\frac{1}{n}}\beta_{\delta}^{\frac{1}{n}}\left( 
\int_\Om (|v_{\delta}|+\delta)^{p+1}\right)^{\frac{p}{n(p+1)}}\cdot |\Om|^{\frac{1-p\eps}{1+p}} \nonumber  \\
& \le & C\lambda^{\frac{1}{n}}\|v_{\delta}\|_{L^{p+1}}\cdot |\Om|^{\frac{1-p\eps}{1+p}}  . \nonumber 
\end{eqnarray}
This implies $\lambda\ge c>0$.  
In \eqref{computation1}, Ko\l{}odziej's $L^\infty$ estimate was used in the first inequality.
In the last inequality,  we have used the estimate $\beta_{\delta}\le C\|v_{\delta}\|_{L^{p+1}}^{n-p}$,
which is due to \eqref{beta1}. 
Hence \eqref{up} was proved for the case $p>n$. The case $p\le n$ follows. by H\"older's inequality.

The second  and third authors of the paper knew the above proof of \eqref{up}  many years ago,
as the proof uses the argument of the third author in \cite{W1}, see also  \cite{W2}.
The proof is in fact quite simple if one is familiar with the argument in \cite{W1}.
The proof given in this paper avoids Ko\l{}odziej's estimate \cite{K98},
for the purpose to provide a PDE proof for the a priori estimates in \cite {WWZ}.
\end{rem}

\vskip 10pt

\section{Brezis-Merle type inequality}

 When studying the Laplace equation 
\begin{equation}\label{Lap-d}
	\begin{cases}
		-\triangle u= f &  \ \text{in $\Omega$,}\\
		u=0 & \ \text{on $\partial\Omega$}
	\end{cases}
	\end{equation}
 with $f\in L^1(\Omega)$ and $\Om\subset \R^2$, Brezis and Merle obtained the following  inequality  \cite{BM}, 
\beq\label{laplace}
\int_\Omega e^{\frac{-(4\pi-\delta)u}{\|f\|_{L^1(\Omega)}}}\,dx\leq \frac{4\pi^2}{\delta}(\text{diam} (\Omega))^2, \
\ \delta\in (0,4\pi),
\eeq
where $u$ is a solution to \eqref{Lap-d}. 
In this section, we study the Brezis-Merle type inequality in high dimensions and the complex Monge-Amp\`ere equation when the right hand term is in the Lorenz-Zygmumd spaces.
First, we have
\begin{lem}\label{WBM}
	Let $\Omega$ be a bounded, smooth, pseudo-convex domain in $\C^n$,
	and $u\in \mathcal{PSH}_0(\Omega)\cap C^{\infty}(\bar \Omega)$.
	Then there exists a constant $\alpha>0$ such that 
	\begin{align}\label{weak-BM}
	\int_\Omega e^{\alpha(-u)}\leq C,
	\end{align}
	if  $\mathcal{M}(u)=1$, 
	where $\alpha, C$ depends on $n$ and $\text{diam}(\Omega)$.
\end{lem}

	\begin{proof}
		If $\Omega$ is a ball, 
		\eqref{weak-BM} follows from \eqref{BMS} and Taylor's expansion.
		For general bounded, smooth, pseudo-convex domain $\Omega\subset\Bbb C^n$ and
		$u\in \mathcal{PSH}_0(\Omega)\cap C^{\infty}(\bar \Omega)$ with $\mathcal{M}(u)=1$,
		we define  
		\begin{align*}
		\Psi:=\begin{cases}
		\det(u_{i\bar{j}})\ \ \ &\text{on}\ \bar\Omega,     \\
		0\ \ \ &\text{on}\ B\setminus\bar\Omega ,
		\end{cases}
		\end{align*}
		where $B$ is a ball containing $\Om$.
		Let $\Psi_k$ be a sequence of smooth, monotone decreasing approximation of $\Psi$,  such that 
		$$ \|\Psi_k-\Psi\|_{L^2(B)}\leq e^{-k}$$
		and $\Psi_k$ converges uniformly to $\Psi$ in $\overline\Om$.
		Let $v_{k}$ be the solution to
		\begin{align*}
		\begin{cases}
		(dd^cv_{k})^n=\Psi_k\,d\mu\ \ \ &\text{in}\ B,  \\
		v_{k}=0\ \ \ &\text{on}\ \p B.
		\end{cases}
		\end{align*}
		Note that by \cite{CP}, $\sup_{B}|v_{k}|$ is uniformly bounded by $\|\Psi\|_{L^2(\Omega_2)}$ and $\displaystyle\lim_{k\to \infty}\int_{B}\Psi_k=1$. By \eqref{weak-BM}, the limit  $v=\displaystyle\lim_{k\to\infty} v_k$ satisfies
		$$\int_B e^{\alpha(-v)}=\lim_{k\to \infty}\int_B e^{\alpha\frac{(-v_k)}{\mathcal{M}(v_k)}}\leq C.$$ 
		Hence by the comparison principle,
		$$\int_\Omega e^{\alpha(-u)}\leq \int_\Omega e^{\alpha(-v)}\leq \int_B e^{\alpha(-v)}\leq C.$$
\vskip-20pt	\end{proof}

We point out that a stronger version of \eqref{weak-BM} was proved in \cite{BB},
where the authors proved the following Brezis-Merle type inequality,
\beq\label{BMeq}
\int_\Omega e^{-nu}\,d\mu\leq A(1-\mathcal M(u))^{-1}
\eeq
if  $\mathcal M(u)<1$.
Here $A>0$ is a constant 
and $\Omega$ is a smooth strictly pseudo-convex domain in $\mathbb C^n$ with $n>1$. 
\eqref{BMeq} implies  the following Brezis-Merle  type inequality
\begin{equation}\label{quasiBM}
\int_\Omega e^{-(n-\delta)u}\,d\mu\leq A\delta^{-(n-1)}
\end{equation}
for $u\in \mathcal{PSH}_0(\Omega)\cap L^\infty(\Omega)$ with $\mathcal M(u)=1$.

Inequalities  \eqref{weak-BM}  and \eqref{quasiBM}  generalize inequality \eqref{laplace} to higher dimensions.
 They can be seen as an analogue of Tian's $\alpha$-invariant \cite{Ti} in domain case, 
 since in a fixed K\"ahler class, the complex Monge-Amp\`ere mass is a constant.  A different proof to \eqref{quasiBM}  is given in \cite{AZ}. 
Here we provide a proof of the weaker form  \eqref{weak-BM} to make the proof of the paper self-contained, 
avoiding the use of the pluri-potential theory.

With Lemma \ref{WBM}, we consider the complex Monge-Amp\`ere equation 
when $f$ is in the Lorenz-Zygmumd spaces. Recall that the Lorenz-Zygmumd space is defined by
\beq\label{LZ}
L^1(\log L)^q(\Omega):=\Big\{f\ |\ \int_\Omega |f|(\log(1+|f|))^q \,dx<\infty\Big\}.
\eeq
It is known that for the Laplace equation 
\eqref{Lap-d} in dimension two, the $L^\infty$-estimate holds when the right hand term $f\in L^1(\log L)^1$. 
For the complex Monge-Amp\`ere equation, it is known that the $L^\infty$-estimate holds 
when $f\in L^1(\log L)^q(\Omega)$ for some $q>n$  \cite{K98}. 
In the following we establish an integral estimate for the case $0<q\leq n$.

For convenience, we denote $A_f=\int_\Omega |f|[\log(1+|f|)]^q \,dx$  for $f\in L^1(\log L)^q$,
and denote by $\mathcal F_q(\Omega)$ the set of pluri-subharmonic functions $u$ for which there exists a sequence of smooth pluri-subharmonic functions $\{u^j\} $ vanishing on $\pom$, such that $u^j\searrow u$ and $\sup_jA_{f_j}<+\infty$, where $f_j=\det(u^{j}_{k\bar l})$.
We have the following generalized Brezis-Merle typed inequalities.
\begin{theo}\label{BMq}
Let $u\in \mathcal F_q(\Omega)$.
\begin{enumerate}
\item  If $0<q<n$, and $\beta=\frac{n}{n-q}$, then there exists $\delta=\delta(\alpha)>0$, where $\alpha$ is defined by \eqref{weak-BM}, and $C=C(\alpha, \beta, A_f)>0$ such that
\beq\label{BMdomain}
\int_\Omega e^{\delta (-u)^\beta} \leq C;
\eeq

\item If $q\ge n$, then  \eqref{BMdomain} holds for any $ \delta, \beta>0$, where $C= C(\alpha, \beta, A_f)>0$.
\end{enumerate}
\end{theo}
\begin{proof}
	        The proof  uses an idea from \cite{CC}.
		We will prove the case $0<q<n$ only, as the proof for the case $q\ge n$ is similar. 
		By \eqref{weak-BM},  \eqref{BMdomain} holds for $\beta<1$ and $\delta>0$. 
		We will use an iteration argument to show that \eqref{BMdomain} holds for $\beta\le \frac{n}{n-q}$.
		
		By definition, it suffices to consider the function $u\in\mathcal F_q(\Omega)\cap C^2(\Omega)$ such that
		\beq\label{CMA0}
		\begin{cases}
			(dd^cu)^n= f & \ \text{in $\Omega$,}\\
			u=0 & \ \text{on $\partial\Omega$.}
		\end{cases}
		\eeq
and $f\in L^1(\log L)^q(\Omega)$.
		Let  $v$ be the solution to
		\beq\label{CMA1}
		\begin{cases}
			(dd^cv)^n=\frac{ f(\log(1+f))^q }{A_f}& \ \text{in $\Omega$,}\\
			v=0 & \ \text{on $\partial\Omega$.}
		\end{cases}
		\eeq
		Let $ G=\epsilon v+(-u)^\beta$, $\beta\geq 1$.  
		We assume $G\geq 0$ on $\Omega$, otherwise, we can restrict to 
		the subdomain  $\{G\geq 0\}$.
		By direct computation
		\begin{eqnarray*}
			u^{i\bar j}G_{i\bar j}&=&\epsilon u^{i\bar j}v_{i\bar j}+u^{i\bar j}[-\beta (-u)^{\beta-1}u_{i\bar j}
			+\beta(\beta-1)(-u)^{\beta-2}u_iu_{\bar j}]\\[5pt]
			&\geq & n\epsilon\Big[\frac{(\log(1+f))^q}{A_f}\Big]^{\frac{1}{n}}-n\beta (-u)^{\beta-1}.
		\end{eqnarray*}
		By Alexandroff's maximum principle, 
		\begin{eqnarray*}
			\sup_{\Omega} G\leq \sup_{\partial\Omega} G
			+C_n\Big(\int_{\Omega}
			\frac{\Big\{\big[ n\epsilon \big(\frac{(\log(1+f))^q}{A_f}\big)^{\frac{1}{n}}-n\beta (-u)^{\beta-1}\big]^-\Big\}^{2n}}{f^{-2}}\Big)^{\frac{1}{2n}},
		\end{eqnarray*}
		where  $[a]^-:=\max\{-a,\,0\}$.
		The integrand is nonzero only if 
		$$n\epsilon \Big(\frac{(\log(1+f))^q}{A_f}\Big)^{\frac{1}{n}}-n\beta (-u)^{\beta-1}<0,$$
		i.e., 
	$$1+f\leq Ce^{\left(\frac{\beta}{\epsilon}\right)^{\frac{n}{q}}\cdot(-u)^{\frac{n(\beta-1)}{q}}}.$$   
		Let $\omega=\big\{z\in\Omega\  |\  n\epsilon \big(\frac{(\log(1+f))^q}{A_f}\big)^{\frac{1}{n}}-n\beta (-u)^{\beta-1}<0\big\}$.
		Note that $G=0$ on $\partial\Omega$. 
		Hence,
		\begin{eqnarray}\label{it-formula}
		\sup_{\Omega} G  &\leq & C_n\Big( \int_\omega(-u)^{2n(\beta-1)}e^{2\left(\frac{\beta}{\epsilon}\right)^{\frac{n}{q}}\cdot(-u)^{\frac{n(\beta-1)}{q}}}\Big)^{\frac{1}{2n}} \nonumber \\[5pt]
		&\leq & C_n\Big( \int_{\Omega}(-u)^{2n(\beta-1)}e^{2\left(\frac{\beta}{\epsilon}\right)^{\frac{n}{q}}\cdot(-u)^{\frac{n(\beta-1)}{q}}}\Big)^{\frac{1}{2n}}.
		\end{eqnarray}   
		Hence, 
		when $\frac{n(\beta-1)}{q}<1$, i.e., when $\beta<\beta_1:=1+\frac{q}{n}$, 
		by \eqref{weak-BM}, we have
		$\sup_{\Omega} G\leq C$; when $\frac{n(\beta-1)}{q}=1$, we choose $\eps>0$ such that $2\left(\frac{\beta}{\epsilon}\right)^{\frac{n}{q}}\le \alpha$, i.e. $\eps\ge \frac{\beta}{(\frac{1}{2}\alpha)^{\frac{q}{n}}}$, where $\alpha$ is defined in \eqref{weak-BM}, we also have $\sup_{\Om}G\le C$.  This implies that 
		for sufficiently $\epsilon>0$ and $\beta\le \beta_1$, 
		$$\epsilon v+ (-u)^\beta\leq C$$
		for some $C>0$. Applying \eqref{weak-BM} to $v$ again, we can choose $\delta$ sufficiently small such that
	\begin{align}\label{beta-k}
	\int_\Omega e^{\delta(- u)^\beta} \leq C\int_\Omega e^{ -\delta \cdot \epsilon v}\leq C
	\end{align}
		for any $\beta\le \beta_1$.

Now we are going to prove inequality \eqref{BMdomain}. 
Let $\beta_{k+1}:=1+\frac{q}{n}\beta_k$, and $\beta_0=1$. Since \eqref{beta-k} have been established for $\beta\le \beta_1$, we repeat the proof and use \eqref{beta-k} in \eqref{it-formula}, which implies that the inequality \eqref{beta-k} holds for $\beta\le \beta_2$.
Once \eqref{beta-k} is established for $\beta\le\beta_k$, where $k\ge 1$, we can repeat the proof again and use \eqref{beta-k} instead of \eqref{weak-BM} in \eqref{it-formula}, and thus \eqref{beta-k} holds for $\beta<\beta_{k+1}$.  Note that when $0<q<n$, the choice of $\delta$ in \eqref{beta-k} is independent of $k$, since we can choose $\eps\ge \frac{n}{(n-q)(\frac{1}{2}\alpha)^{\frac{q}{n}}}$ and  $\delta\le \frac{\alpha}{\eps}\le C(n,q)\alpha^{1+\frac{q}{n}}$. Hence, by the iteration argument, it can be improved for $\beta\le \displaystyle\lim_{k\to \infty}\beta_k=\frac{n}{n-q}$.
	\end{proof}


\section{Proof of Theorem \ref{mt}}
In this section, we prove the Moser-Trudinger type inequality \eqref{mt} by means of Theorem \ref{BMq}.

\begin{theo}
	Let $\Omega$ be a bounded, smooth, pseudo-convex domain in $\mathbb{C}^n$.
	There exist positive constants $\alpha,\ C>0$, depending only on $n$  and $\text{diam}(\Omega)$, 
	such that  $\forall\ u\in \mathcal{PSH}_0(\Omega)\cap C^{\infty}(\bar \Omega)$, $ u\not\equiv 0$,
	we have the inequality
	\beq\label{MT}
	\int_\Omega e^{ \alpha \left(\frac{- u}{\|u\|_{\mathcal{PSH}_0(\Omega)}}\right)^{\frac{n+1}{n}}} \leq C . 
	\eeq
\end{theo}
\noindent{\it Proof.}
Let $m\ge n$, $q=\frac{n}{n+1}$, $\beta=\frac{n}{n-q}=\frac{n+1}{n}$ and $\alpha<\delta$, where $\delta$ is defined in Theorem \ref{BMq}. 
Note that $\beta q=1$. We denote $\|u\|=\|u\|_{\mathcal{PSH}_0(\Omega)}$ for simplicity.
To derive the inequality, we  consider the approximation
$$\F_m(u,\Om):=\int_\Om F_m\left(\frac{-u}{\|u\|}\right),$$
where $$F_m(t):=\sum_{j=n}^m\frac{\alpha^j}{j!}t^{j\beta},\ \ \ f_m(t):=F_m'(t)=\beta\sum_{j=n}^m\frac{\alpha^j}{(j-1)!}t^{j\beta-1}.$$
By the Sobolev inequality, $Y_m(\Omega) :=\sup\{\F_m(u, \Omega)\}<+\infty$ for each $m$. 
It suffices to prove that $Y_m(\Omega)$ is uniformly bounded for $m$.

Similarly as in Section 2, we consider the modified functional
\beq\label{mf}
\F_{m,\delta,\eta}(u):=\int_\Om F^\delta_m\left(\frac{-u}{\eta(\|u\|)}\right),
\eeq
where $F_m^\delta(t)= F_m(t)+\delta t$, $f_m^\delta(t)=f_m(t)+\delta$ and $\eta(t)=e^{t-1}$. 
Then by direct computations,  the gradient flow of $\F_{m,\delta,\eta}$ is
\begin{align}\label{quoti-grad}
	\begin{cases}
		u_t=\log \det (u_{i\bar{j}})-\log \{\lambda f^\delta_m\left(\frac{-u}{\eta(\|u\|)}\right)\},  \\[4pt]
		u(x,0)=u_0,
	\end{cases}
\end{align}
where for any $u_0\in C^\infty(\Om)\cap\mathcal{PSH}_0(\Om)$, and
$$\lambda=\frac{(n+1)\cE(u)}{\int_\Om (-u)f^\delta_m\left(\frac{-u}{\eta(\|u\|)}\right)}.$$

By the a priori estimates in Theorem \ref{para-estimate}, 
There is a long time solution to  equation \eqref{quoti-grad}, 
which converges to a smooth maximizer $u_{m,\delta}$ of 
$Y_{m,\delta,\eta}(\Omega)=\sup\{\F_{m,\delta,\eta}(u)\}$,
and satisfies 
\beq\label{quoti-ell}
\begin{cases}
	\det (u_{i\bar j})=\lambda f^\delta_m\left(\frac{-u}{\eta(\| u\|)}\right)=:g_{m,\delta}
	& \text{in $\Omega$,} \\[5pt]
	u=0  & \text{on $\partial \Omega$}.
\end{cases}
\eeq
By the following lemma, we have furthermore $\|u_{m,\delta}\|=1$.

\begin{lem}\label{norm}
Denote 
\begin{align*}
\Theta^*(\eps):=&\sup\{\|u\|
			\,\big|\,u\in \mathcal{PSH}_0(\Omega),\,\F_{m,\delta,\eta}\ge Y_{m,\delta,\eta}-\eps\},   \\
\Theta_*(\eps):=&\inf\{\|u\|\,\big|\,u\in \mathcal{PSH}_0(\Omega),\,\F_{m,\delta,\eta}\ge Y_{m,\delta,\eta}-\eps\}.  
\end{align*}
Then 
$$\Theta^*(\eps),\,\Theta_*(\eps)\to 1,\ \ \ \text{as }\eps\to 0.$$
\end{lem}
\begin{proof}
Denote $g(t)=\frac{t}{\eta(t)}=te^{1-t}$. Then $g(t)<g(1)=1$ for any $t\neq 1$. For any $t>0$, $t\neq 1$ and any $u\in\mathcal{PSH}_0(\Omega)$ with $\|u\|=t$, we have, 
$$\F_{m,\delta,\eta}(u)=\int_{\Om}F_{m}^{\delta}\left(\frac{u}{\eta(\|u\|)}\right)=\int_{\Om}F_{m}^{\delta}\left(g(t)\frac{u}{\|u\|}\right).$$
By the Taylor expansion of $F_{m}^{\delta}$,  
$$\F_{m,\delta,\eta}(u)\le g(t)\int_{\Om}F_{m}^{\delta}\left(\frac{u}{\|u\|}\right).$$
Therefore, 
$$\sup\{\F_{m,\delta,\eta}(u)\,\big|\,u\in\mathcal{PSH}_0(\Omega),\,\|u\|=t\}\le g(t)Y_{m,\delta,\eta}.$$
Lemma \ref{norm} follows immediately. 
\end{proof}

Lemma \ref{norm} was first proved in \cite{TiW} for the $k$-Hessian equation.
Now it suffices to prove 
\beq\label{ec}
\int_{\Om}e^{\alpha(-u_{m,\delta})^{\beta}}\, d\mu\le C 
\eeq
for constant $C>0$,  uniformly bounded as $m\to \infty$ and $\delta\to 0$.
It  implies that $Y_{m,\delta,\eta}$ are uniformly bounded. 

We claim $g_{m,\delta}\in L^1 (\log L^1)^q$, i.e. $$\int_{\Om}g_{m,\delta}(\log (1+g_{m,\delta}))^q\,d\mu\le C.$$
Then \eqref{ec} follows from Theorem \ref{BMq}. 
To prove the claim, we denote for simplicity that $v=u_{m,\delta}$. 
By definition, 
\begin{eqnarray*}
	f^\delta_m(-v)&= & \alpha\beta\sum_{j=n}^m  \frac{\alpha^{j-1}}{(j-1)!} (-v)^{(j-1)\beta+\beta-1}+\delta\\
&\leq& \alpha \beta e^{\alpha(-v)^\beta}(-v)^{\beta-1}-\alpha\beta\sum_{j=0}^{n-2}\frac{\alpha^j}{j!}(-v)^{j\beta+\beta-1}+\delta.
\end{eqnarray*}
We may assume there is a subsequence $m_j\to \infty$ such that $\int_\Om(-v)f^\delta_{m_j}(-v)>1$ and thus $\lambda$ is bounded,
otherwise the proof has been finished. 
Noting that $$\log (1+g_{m,\delta}(-v))\le C+\alpha (-v)^{\beta}$$ and $\beta q=1$ on the subset $A:=\{-v>1\}$ and $f_m^\delta$ is bounded near boundary, we have
\begin{eqnarray*}
	\int_\Omega g_{m,\delta}(\log(1+g_{m,\delta}))^q &\le& \frac{\int_{A}f_m^\delta(-v)\cdot(C+\alpha(-v)^\beta)^{q}}{\int_\Om(-v)f_m^\delta(-v)}+\frac{\int_{\Om\setminus A} f_m^\delta(-v)\log (1+g_{m,\delta}(-v))}{\int_{\Om}(-v)f_m^\delta(-v)}\\ 
&\leq& C.
\end{eqnarray*}
This completes the proof. 
\hfill$\square$

\vskip 20pt

The Sobolev and Moser-Trudinger type inequalities for the $k$-Hessian equations were proved in \cite{W1, TiW}.
A natural question is whether one can extend those inequalities to the complex setting. 
In this paper we obtained these inequalities but we haven't obtained the optimal constants. 
A key technique used in \cite{W1, TiW} is the moving plane method, 
which implies the rotational symmetry of solutions to  the Dirichlet problem in the unit ball. 
But this technique does not apply to the complex Monge-Amp\`ere equation.



\begin{rem}
Note that by the stability result \cite{CP}, Theorem \ref{MT} also holds for bounded hyper-convex domains. 
Indeed, for any bounded hyper-convex domain $\Omega\subset \Bbb C^n$,
there exists a smooth pluri-subharmonic function $u$ vanishing on $\pom$. 
Denote $f=\det\{u_{i\bar{j}}\}$. 
Let $\Om_j\uparrow\Om$ be an increasingly pseudo-convex domains with smooth boundaries, 
and let $u_j \in C^{\infty}(\overline\Om)\bigcap \mathcal{PSH}_0(\Om)$ be the unique solution to $ \det(u_{i\bar{j}})=f$ in $\Om_j$.
By the stability result on $\Om$, $u_j$ converges uniformly to $u$ and the inequality holds by taking limits. 
\end{rem}

\end{document}